\documentclass[a4paper,10pt]{article}
\usepackage{a4}
\usepackage{amsthm,amsmath,amssymb,setspace,authblk}
\newtheorem{theorem}{Theorem}[section]

\newtheorem{cor}[theorem]{Corollary}

\newtheorem{rem}[theorem]{Remark}
\newtheorem{conj}[theorem]{Conjecture}

\vspace{5cm}
\begin{document}
\title { {\Large \bf{\sc Existence of  primes between two consecutive squares}}}
\author[]{M. Sundarakannan}
\affil[]{\small{Department of Mathematics, SSN College of Engineering, Chennai - \mbox{ 603 110,} \mbox{Tamil Nadu}, India. email: m.sundarakannan@gmail.com}}
\renewcommand\Authands{ and }
\maketitle
\begin{abstract}
Legendre's Conjecture is one of the most elegant open problems in Number Theory, which states that there is a prime between consecutive two perfect squares.  In this note, we prove the conjecture holds true and also discuss the related results.

\end{abstract}



\bigskip
\noindent

\section{Introduction}
 The fundamental theorem of arithmetic shows prime numbers are the building blocks of integers.   In number theory,  there are many open problems regarding the existence of prime numbers in given intervals.   In 1912,  Edmund Landau listed some of the fundamental problems related to prime numbers. One of the Landau's listed problems is Legendre's conjecture which states that:

\begin{conj}{\rm \cite{book1, book2}({\rm Legendre's conjecture })}\label{conj1}
There is a prime number between consecutive two perfect squares $?$.
\end{conj}

A few other conjectures also exist in relation to Conjecture~\ref{conj1}.  These are:
\begin{conj}{\rm \cite{book1, opp}({\rm Oppermann's conjecture })}\label{conj3}
For every integer $n > 1$, there is at least one prime number between
$(n(n - 1), n^2)$ and at least another prime number between $(n^2, n(n+1))$
\end{conj}

\begin{conj}{\rm \cite{andri}({\rm Andrica's conjecture })}\label{conj2}
For every pair of consecutive prime numbers $P_n$ and $P_{n+1}$ such that $\sqrt P_{n+1} - \sqrt P_n < 1$.
\end{conj}

\begin{conj}{\rm Pg. 22 \cite{book1}({\rm Brocard's conjecture})}\label{conj4}
There are at least four primes between the squares of any two consecutive primes, with the
exception of 2 and 3.
\end{conj}

For the subsequent propositions the following results are essential.
\begin{theorem}{\rm \cite{book2}}
 The linear Diophantine equation (LDE) $a x + b y = c$ is solvable if and only if $d | c$, where $d= GCD (a, b)$.  If $x_0 , y_0$ is a particular solution of the LDE, then all its solutions are given by
 \begin{equation*}
 x = x_0 + \Big(\frac{b}{d}\Big)t \,\,\, {\rm and} \,\,\,\, y = y_0 - \Big( \frac{a}{d} \Big)t
 \end{equation*} where $t$ is an arbitrary integer.
\end{theorem}

\begin{theorem}{\rm \cite{book2}}\label{thm2}{\rm ( Pythagorean )}
\begin{enumerate}
\item The sum and product of  any two even integers are even.
\item Then sum of any two odd integers is even.
\item The product of any two odd integers is odd.
\item The sum of an even integer and an odd integer is odd.
\item The product of an even integer and an odd integer is even.
\end{enumerate}
\end{theorem}

In this note, we prove the above mentioned conjectures hold true.
\section{Main Results}
\noindent
{\bf Proof of Conjecture \ref{conj1}} and \ref{conj3}.
\begin{theorem}\label{thm1}
There exist  primes between any two consecutive perfect squares.
\end{theorem}
\begin{proof}
Let $x$ and $y$ be two distinct odd positive integer in $(n^2, (n+1)^2)$ such that $x + y = 2 n(n+1)$.  Consider $x = n^2 + k$ and $y = n^2 + l$ for some $k , l \in \mathbb{Z}$.  Now $x + y = 2 n^2 + k + l = 2 n (n+1)$ gives
\begin{equation}\label{eq1}
k + l = 2 n.
\end{equation}
Solving the LDE (\ref{eq1}), we have $k = n - t$ and $l = n + t$ for all $t \in \mathbb{Z}$.  Since $x$ and $y$ are odd integer and \mbox{$n^2 < x,y < (n+1)^2$},  $t$ must be odd and $1 \leq t < n$. Hence
\begin{equation*}\label{eq2}
 x = n^2 + n - t \,\,\in (n^2, n(n+1)) \,\,\,\, {\rm and}
  \end{equation*}
\begin{equation}\label{eq2}
    y = n^2 + n + t  \,\,\,\in (n(n+1), (n+1)^2)\,\,\,{\rm for \,\, some}\,\, t\in 2 \mathbb{Z} + 1
\end{equation}
The list of odd integer in the interval $(n^2, (n+1)^2)$ is of the form $c \pm 1,\, c \pm 3,\, \dots , c \pm t$ where $c = n(n+1)$ and odd $t < n$.
\noindent
Now to prove either $x$ or $y$ is prime.  Based on (\ref{eq2}),
 \begin{equation}\label{eq3}
  y - x = 2 t \,\,\, {\rm for \,\,some}\,\, t \in 2 \mathbb{Z} + 1 .
 \end{equation}
 Solving the LDE (\ref{eq3}) we have $x = -( t + t')$ and $y = t - t'$ for all $t' \in \mathbb{Z}$.  Since $x$ and $y$ are odd, $t'$ must be even.  By Theorem \ref{thm2},  $t'$ is sum of any two odd integer. Now consider $t'$ is sum of $t$ with  odd positive integer $p$ (say). The possible $t'$ are $\pm  t \pm p \in 2 \mathbb{Z}$ :

  \begin{itemize}
  \item if $t' = -t + p$, then $x$ is negative, which is contradiction to the assumption $x$ and $y$ are positive.
  \item if $t' = t + p$, then $x$ and $y$ are negative, which is again contradiction to the assumption $x$ and $y$ are positive.
         \item if $t' = -t - p$, then $x = p $ and $y= 2 t + p$.
    \item if  $t' = t - p$, then $x = p - 2t$ and $y = p $.
   \end{itemize}

   Let $\mathbb{X} = \{ \,p \,\in \,\mathbb{Z}\, | \,\,p \,\,{\rm is \,\, odd} \}$. Since $p$ is odd, there exist a odd prime $p$ must be in $\mathbb{X}$. Based on $t' = \pm t - p$, we conclude there exist a prime in $(n^2, n(n+1))$ and $(n(n+1), (n+1)^2)$.
 \end{proof}

\begin{cor}\label{cor1}
Any prime number $p \geq 5$ can be written as either $n^2 + n -t$ or $n^2 + n + t$ where $n = \lfloor \sqrt p \rfloor$ and   odd  $ t < n$.
\end{cor}

\begin{rem}
{\rm Based on Theorem~\ref{thm1}, Conjecture~\ref{conj4} holds true, since the minimum prime gap is two.  That is, if $P_1$ and $P_2$ are two consecutive prime numbers with $P_2 > P_1 > 2$, then there exist at least four prime numbers between $(P_1^2, P_2^2)$.}
\end{rem}

\noindent
{\bf Proof of Conjecture \ref{conj2}}.

Let $P_N$ be $N^{th}$ prime number. By Corollary  \ref{cor1}, any prime number is of the form $P_N = n^2 + n \pm t$ where $n = \lfloor \sqrt{P_N }\rfloor$ and   odd  $ t < n$.

 \smallskip
 If $P_N, P_{N+1} \in (n^2, (n+1)^2)$, then clearly for all $n$ and  $ 1 \leq t, t' < n$,
  \begin{equation*}
 \sqrt {P_{N+1}} - \sqrt {P_N } = \sqrt {n^2 + n \pm t } - \sqrt {n^2 + n \pm t'} < 1.
 \end{equation*}
\noindent
Suppose $P_N \in (n^2, (n+1)^2)$ and $P_{N+1} \in ((n+1)^2, (n+2)^2)$. For all positive integer $n$ and  $ 1 \leq t < n$ \,\, \& \,\, $ 1 \leq  t' < n+1$,
\begin{align*}
&(\sqrt {P_{N+1}})^2 - (1+ \sqrt {P_N })^2 = P_{N+1} - (1 + 2 \sqrt {P_N} + P_N)\\
& = (n+1)^2 + (n+1) - t' - (1 + 2 \sqrt {(n^2 + n + t)} + (n^2 + n + t) )\\
& = 2n + 1 - (t' + t) - 2 \sqrt {n^2 + n + t} \,\, < 0.
\end{align*}  Hence, $\sqrt {P_{N+1}} - \sqrt {P_N} < 1$.

\end{document}